\documentclass[11pt]{amsart}
\usepackage[margin=1in]{geometry}
\usepackage{amsthm}
\numberwithin{equation}{section}
\usepackage{amsmath}
  \usepackage{paralist}
  \usepackage{graphics} %% add this and next lines if pictures should be in esp format
  \usepackage{epsfig} %For pictures: screened artwork should be set up with an 85 or 100 line screen
\usepackage{graphicx}  \usepackage{epstopdf}%This is to transfer .eps figure to .pdf figure; please compile your paper using PDFLeTex or PDFTeXify.
 \usepackage[colorlinks=true]{hyperref}
\hypersetup{urlcolor=blue, citecolor=red}

\newcommand*{\avint}{\mathop{\ooalign{$\int$\cr$-$}}}

\newcommand{\ib}{\int_{B_r(y)}}

\newcommand{\mnp}{(m\cdot\nabla p) }

\newcommand{\ot}{\Omega_T }

\newcommand{\RN}{\mathbb{R}^N}
\newcommand{\prt}{p_{y,r}(t) }

\newtheorem{theorem}{Theorem}[section]

\newtheorem*{main}{Main Theorem}
\newtheorem{lemma}[theorem]{Lemma}
\newtheorem{proposition}{Proposition}[section]

\theoremstyle{definition}
\newtheorem{definition}[theorem]{Definition}
\newtheorem{remark}{Remark}

\newcommand{\ep}{\varepsilon}
\newcommand{\eps}[1]{{#1}_{\varepsilon}}

\title[a biological network formulation model
] %Use the shortened version of the full title
      {Regularity theorems for a biological network formulation model in two space dimensions}

\author[Xiangsheng Xu]{}

 \keywords{$A_q$-weights, biological network formulation, bounded mean oscillation (BMO), weakly monotone functions. {\it Kinet. Relat. Models}, to appear. }

 \email{xxu@math.msstate.edu}
\begin{document}
\maketitle

% Enter the first author's name and address:
\centerline{\scshape Xiangsheng Xu}
\medskip
{\footnotesize
% please put the address of the first author
 \centerline{Department of Mathematics \& Statistics}
   \centerline{Mississippi State University}
   \centerline{ Mississippi State, MS 39762, USA}
} % Do not forget to end the {\footnotesize by the sign }

\bigskip

% The name of the associate editor will be entered by an editorial staff
% "Communicated by the associate editor name" is not needed for special issue.
% \centerline{(Communicated by the associate editor name)}

%The abstract of your paper
\begin{abstract}
We present several regularity results for a biological network formulation model originally introduced
by D. Cai and D. Hu \cite{HC}. A consequence of these result is that a stationary weak solution must be a classical one in two space dimensions. Our mathematical analysis is based upon the weakly monotone function theory and Hardy space methods. 
\end{abstract}
%The title of your section 1

	\section{Introduction}
	Let $\Omega$ be a bounded domain in $\mathbb{R}^N$ and $T$ a positive number. Set $\ot=\Omega\times(0,T)$. We study the behavior of weak solutions of the system
	\begin{align}
	-\mbox{{div}}\left[(I+m\otimes m)\nabla p\right]&=S(x)\ \ \ \mbox{in $\ot$,}\label{e1}\\
	\partial_tm-D^2\Delta m-E^2\mnp\nabla p+|m|^{2(\gamma-1)}m&=0\ \ \ \mbox{in $\ot$,}\label{e2}
	\end{align}
	coupled with the initial boundary conditions
	\begin{align}
	p(x,t)=0, &\ \  m(x,t)=0, \ \ \ (x,t)\in \Sigma_T\equiv\partial\Omega\times(0,T),\label{e3}\\
	m(x,0)&=m_0(x),\ \ \ \ x\in\Omega\label{e4}
	\end{align} 
	for given function $S(x)$ and physical parameters $D, E, \gamma$ with properties:
	\begin{enumerate}
		\item[(H1)] $S(x)\in L^{q_0}(\Omega), \ q_0>\frac{N}{2}$; and
		\item[(H2)] $D, E\in (0, \infty), \gamma\in (\frac{1}{2}, \infty)$.
	\end{enumerate}
	This system was originally derived in (\cite{H}, \cite{HC}) as the formal gradient flow of the continuous version of a cost functional describing formation of biological transportation networks on discrete graphs. In this context, the scalar function $p=p(x,t)$ is the pressure due to Darcy's law, while the vector-valued function $m=m(x,t)$ is the conductance vector. The function $S(x)$ is the time-independent source term. Equation \eqref{e1} can be interpreted as Kirchhoff's law for the flux $u\equiv-(I+m\otimes m)\nabla p$,  and the cost is proportional to $|(u\cdot\nabla p)|+c|m|^{2\gamma}$, where $c$ is a constant\cite{HMPS}.
	%has been proposed by Hu and Cai (\cite{H}, \cite{HC}) to describe 
	%natural network formulation. Then 
	Values of the parameters $D, E$, and $\gamma$ are
	determined by the particular physical applications one has in mind. To give an example, we have $\gamma =\frac{1}{2}$ for blood vessel systems. We would like to refer the reader to \cite{HMPS} for more discussions
	in this regard.
	% corresponds to leaf venation \cite{H}. 
	
	In general nonlinear problems do not possess classical solutions. A suitable notion of a weak solution must be obtained for \eqref{e1}-\eqref{e4}. It turns out \cite{HMP} that we can introduce the following definition.

	%	\noindent {\bf Definition.} 
	\begin{definition}
		A pair $(m, p)$ is said to be a weak solution if:
		\begin{enumerate}
			\item[(D1)] $m\in L^\infty\left(0,T; \left(W^{1,2}_0(\Omega)\cap L^{2\gamma}(\Omega)\right)^N\right),\ \partial_tm\in L^2\left(0,T; \left(L^2(\Omega)\right)^N\right),\  p\in L^\infty(0,T; W^{1,2}_0(\Omega)), \ m\cdot\nabla p \in L^\infty(0,T;  L^{2}(\Omega))$;
			\item[(D2)] $m(x,0)=m_0$ in $C\left([0,T]; \left(L^2(\Omega)\right)^N\right)$;
			\item[(D3)] Equations \eqref{e1} and \eqref{e2} are satisfied in the sense of distributions.
		\end{enumerate}
	\end{definition}
	A result in \cite{HMP} asserts that \eqref{e1}
	-\eqref{e4} has a weak solution provided that, in addition to assuming $S(x)\in L^2(\Omega)$ and (H2), we also have
	\begin{enumerate}
		\item[(H3)] $m_0\in\left( W^{1,2}_0(\Omega)\cap L^{2\gamma}(\Omega)\right)^N$.
	\end{enumerate} 
	Finite time extinction or break-down of solutions in the spatially one-dimensional setting for certain ranges of the relaxation exponent $\gamma$ was carefully studied in \cite{HMPS}. Further modeling analysis and numerical results can be found in \cite{AAFM}. We also mention that the question of existence in the case where $\gamma=\frac{1}{2}$ is addressed in \cite{HMPS}. In this case the term $|m|^{2(\gamma-1)}m$ is not continuous at $m=0$. %In this case the term  $|m|^{2(\gamma-1)}m$ 
	It must be replaced by the following function %$g$ 
	$$g(x,t)=\left\{\begin{array}{ll}
	|m|^{2(\gamma-1)}m & \mbox{if $m\ne 0$,}\\
	\in [-1,1]^N &\mbox{if $m\ne 0$.}
	\end{array}\right.$$
	%The existence of weak solutions of this initial boundary value problem was proved by Haskovec, Markowich, and Perthame \cite{HMP}. 
	However, the general regularity theory remains fundamentally incomplete. In particular, it is not known whether or not weak solutions develop singularities in high space dimensions. Recently, Jian-Guo Liu and the author \cite{LX} obtained a partial regularity theorem for \eqref{e1}-\eqref{e4}. It states that the parabolic Hausdorff dimension of the set of singular points can not exceed $N$, provided that $N\leq 3$.
	
	In this paper, we continue to study the regularity properties of weak solutions. We focus our attention on the case where $N=2$. Our main result is:
	\begin{main} Let  (H1) and (H2) be satisfied, and
		let $(m, p)$ be a stationary weak solution to \eqref{e1}-\eqref{e3}, i.e., the functions $m$ and $p$  are independent of time. Assume that the space dimension $N$ is $2$.
		Then $(m, p) $ is locally a  classical solution, provided that $S(x)$ is locally H\"{o}lder continuous.
	\end{main}
	
	If we apply the proof of the partial regularity theorem in \cite{LX} to the situation considered here, we can only conclude that the set of singular points is countable. Even though our estimates are interior ones, we do not foresee any major difficulty in extending our results to
	the boundary. We encourage the interested reader to try that. 
	%The title of your section 2
	
	We begin by studying the time-dependent problem. A key observation is that the function $p$ can be decomposed into two pieces in a small neighborhood: The first piece is weakly monotone \cite{M}. A result in \cite{M} asserts that a weakly monotone function in $W^{1,N}_{\mbox{loc}}(\Omega)$ is locally continuous. The proof of this result becomes applicable to our case if $N=2$. The second piece is bounded due to a result in \cite{LX}. The combination of two is enough to yield the local continuity of $p$ in the space variables. This result is then used to prove that $ |m|^\beta\in L^2(0,T; W^{1,2}_{\mbox{loc}}(\Omega)),\ |m|^\beta|\nabla p|^2\in L^1(0,T; L^1_{\mbox{loc}}(\Omega))$ for each $\beta>0$. Thus in the two-dimensional stationary case, 
	%The second observation is that 
	$|m|^\beta$ can be viewed as a
	BMO function for each $\beta>0$
	% in the space variables
	(see Section 2 for definition and other relevant information). This combined with a result in \cite{JN} asserts that $(|m|^2+1)^\sigma$ is an $A_q$-weight for each $ q>1, \sigma>0$. Equipped with this, we are able to establish 
	%This leads to the establishment 
	that $\mnp^2, |\nabla p|^2$ both lie in 
	the local Hardy space, 
	%for each $t$,  in the time-independent case
	from whence follows the local continuity of $m$.

	%We would like to remark that it would have seemed natural for us to try to apply a weak, weighted version of  here. Unfortunately, it does not work. The reason is that 
	To describe the mathematical difficulty involved in our problem, first notice the term $\mnp \nabla p$ in \eqref{e2}, which represents a cubic nonlinearity. Currently, there is little work done on this type of nonlinearities. Second, the elliptic coefficients in \eqref{e1}
	satisfy
	$$|\xi|^2\leq \langle (I+m\otimes m)\xi,\xi\rangle\leq (1+|m|^2)|\xi|^2\ \ \ \mbox{for all $\xi\in \mathbb{R}^N$,}$$
	where $\langle \cdot,\cdot\rangle$ denotes the inner product in $\mathbb{R}^N$.
	%The mathematical difficulty of the problem is due to the fact that 
	%That is to say, 
	Since $m$ is not bounded a priori,
	%That is, 
	the largest eigenvalue $\lambda_l$ and the smallest eigenvalue $\lambda_s$ of the coefficient matrix may not satisfy 
	$$\lambda_l\leq c\lambda_s.$$
	Here and in what follows the letter $c$ denotes a generic positive number.
	%is not bounded by a constant times the smallest one.  %$1+|m|^2$, while the smallest one $\lambda_s$  is $1$. As %far as we know, existing regularity theory  only deals with the case where there is a positive number $c$ such
	%\begin{equation}\label{elc}
	%\lambda_l\leq c\lambda_s
	%\end{equation}
	%and $\lambda_s$ is an $A_2$ weight. 
	%these two eigenvalues belong to similar classes of weights. 
	%See, e.g., \cite{HKM}. 
	Thus existing results for degenerate and/or singular elliptic equations such as these in \cite{HKM} are not applicable. 
	%In our case, \ref{elc} is what we wish to establish. 
	A condition in \cite{CW} seems to 
	be satisfied by our elliptic coefficients, but the results there cannot be used to improve the regularity of the terms $|\nabla p|^2,\ \mnp^2$, neither are we able to employ a weighted
	version of
	%this regard, the connection between $1$ and $1+|m|^2$ can not be established, which 
	%This explains why we have not been able to employ
	%in order for
	% As a result, 
	%well known weighted inequalities such as weighted Sobolev's inequality and 
	Gehring's lemma \cite{K} in our analysis. Instead, we are motivated by an idea from \cite{CLMSN}. See Proposition \ref{pro3} for details.
	%to play any useful roles.
	
	%They do not match well in weighted inequalities. This basically rules
	%out the possibility of applying 
	
	The rest of the paper is organized as follows. In Section 2, we collect some relevant results about maximal functions, BMO functions, Hardy spaces, and $A_q$-weights. Various regularity results are presented in Section 3. Our main theorem is a consequence of these results. 
	\section{Preliminary results} In this section, we review some relevant results about maximal functions, BMO functions, Hardy spaces, and $A_q$ weights.
	
	Consider the Hardy-Littlewood maximal function $Mf$ of a given measurable function $f$, which is defined by
	\begin{equation}\label{md1}
	Mf(y)=\sup_{r>0}\avint_{B_r(y)} |f|dx,
	\end{equation}
	%$$$
	where $\avint_{B_r(y)} |f|dx=\frac{1}{|B_r(y)|}\ib |f(x)|dx$, $B_r(y)$ denotes the ball in $\mathbb{R}^N$ with center at $x$ and radius $r$, and $|B_r(y)|$ is its Lebesgue measure. Also, if no confusion arises, we always suppress the dependence of a function on its independent variables. We refer the reader to \cite{S1} for the full story of maximal functions. Here we only
	mention the well known inequality
	\begin{equation}
	\|Mf\|_q\leq c(q)\|f\|_q\ \ \ \mbox{when $q>1$,}\label{md}
	\end{equation}
	where $\|\cdot\|_q$ is the norm in $L^q(\mathbb{R}^N)$. 
	%Throughout this paper the letter $c$ denotes a generic positive constant whose meaning can clearly be understood from the context. 
	Note that \eqref{md} fails when $q=1$.
	
	Define a class $\mathcal{T}$ of normalized test functions on $\mathbb{R}^N$ by
	$$\mathcal{T}=\{\phi\in C^\infty(\mathbb{R}^N): \mbox{supp}\ \phi\subset B_1(0)\ \mbox{and $\|\nabla\phi\|_\infty\leq 1$}\}.$$
	Define the ``grand maximal function'' $f^*$ of a distribution on $\mathbb{R}^N$ by
	\begin{eqnarray}
	f^*(y)&=&\sup_{r>0}\sup_{\phi\in\mathcal{T}}\left|\int\frac{1}{r^N}\phi(\frac{y-x}{r})f(x)dx\right|\nonumber\\
	&=&\sup_{r>0}\sup_{\phi\in\mathcal{T}}\left|\phi_r*f\right|.
	\end{eqnarray}
	Here we write $\phi_r$ for the function $r^{-N}\phi(r^{-1}y)$. Note the similarity between this and \eqref{md1}. In particular, $f^*\leq cMf$, and $Mf\leq cf^*$ if $f\geq 0$. A distribution $f$ in $\mathbb{R}^N$ is said to lie in the Hardy space $\mathcal{H}^1(\mathbb{R}^N)$ if $f^*\in L^1(\mathbb{R}^N)$, and the Hardy space norm is
	defined by
	$$\|f\|_{\mathcal{H}^1}=\|f^*\|_1.$$
	There is an alternative definition to this that is equivalent and simpler. Specifically, if $\phi$ is any $C^\infty$ function on $\mathbb{R}^N$ with compact support and $\int_{\mathbb{R}^N}\phi dx>0$, then $f$ lies in
	$\mathcal{H}^1(\mathbb{R}^N)$ if and only if 
	$$\sup_{r>0}\left|\phi_r*f\right|\in L^1(\mathbb{R}^N).$$
	For the purpose of applications to boundary value problems for PDE, we need a local version of the Hardy space.
	\begin{definition}
		Let $\Omega$ be an open set in $\mathbb{R}^N$. We say that a distribution $f$ on $\Omega$ lies in $\mathcal{H}^1_{\mbox{loc}}(\Omega)$ if for each compact set $K_0\subset\Omega$ there is an $\varepsilon_0>0$ so that
		$$\int_{K_0}\sup_{0<r<\varepsilon_0}\sup_{\phi\in\mathcal{T}}|\phi_r*f|dx<\infty.$$
	\end{definition}
	\begin{lemma}
		Suppose $f\in L^1_{\mbox{loc}}(\Omega)$. Then $f\in \mathcal{H}^1_{\mbox{loc}}(\Omega)$ if $|f|\ln(2+|f|)\in L^1_{\mbox{loc}}(\Omega)$, and the converse is true when $f\geq 0$.
	\end{lemma}
	This lemma can be found in \cite{S1}.
	
	The definition of $BMO$ is that $f\in BMO$ if
	$$\sup_{B_r(y)}\avint_{B_r(y)}|f-f_{y,r}|dx\equiv\|f\|_*<\infty,$$
	where $f_{y,r}=\avint_{B_r(y)} fdx$. Of course, bounded functions are in $BMO$ and $\ln|x|$ is an unbounded function in $BMO$. A closely related subject is the one of $A_q$-weights. A locally integrable, non-negative function $w$ on $\mathbb{R}^N$ is said to lie in $A_q$, where $q\in (1,\infty) $, if
	% which are defined by the condition
	$$A_q(w)\equiv\sup_{B_r(y)}\left(\avint_{B_r(y)} w dx\left(\avint_{B_r(y)} w^{-\frac{1}{q-1}}dx\right)^{q-1}\right)<\infty.$$
	It turns out that the inequality
	$$\int_{\mathbb{R}^N}[Mf(x)]^q w(x)dx\leq c\int_{\mathbb{R}^N}|f(x)|^q w(x)dx$$
	holds for each $f\in L^q(\mathbb{R}^N)$ if and only if $w\in A_q$. Also of interest are $A_1=\{w:Mw(x)\leq cw(x)\ \mbox{for some constant $c$ and all $x\in \RN$}\}$ and $A_\infty=\cup_{q>1}A_q$. We have the set inclusions
	$$A_1\subseteq A_{q_1}\subseteq A_{q_2}\subseteq A_{\infty},$$
	where $1\leq q_1\leq q_2\leq \infty$.
	The following result is contained in \cite{JN}.
	\begin{lemma} \label{jnl}Set
		$$BMO_*=\{w: w\geq 0, w, \frac{1}{w}\in BMO\}.$$
		Then $BMO_*\subset \cap_{q>1}A_q$.
		% We have:
		%	\begin{enumerate}
		%		\item[(R1)] if $u\in BMO$ and $\alpha>0$, then $u^2+\alpha\in$;
		%		\item[(R2)] both $w$ and $\frac{1}{w}$ lie in $\cap_{q>1}A_q$ if and only if $\ln w$ is in the closure
		%		of $L^\infty(\mathbb{R}^N))$ in $BMO$.
		%	\end{enumerate}
	\end{lemma}
	
	This lemma has played a key role in the proof of our main result.
	
	It is well known that there are many useful results for $L^q$ spaces when $1<q<\infty$ that fail when $q=1$ or $q=\infty$. This happens, for example, when one is faced with the equation $\Delta u=g$ and one wants to have $L^q$ estimates for $\nabla^2u$ in terms of $g$. Hardy spaces provide alternatives to $L^q$ when $q=1$ for which there are counterparts to the familiar estimates for $1<q<\infty$. In general $BMO$ is the right substitute for $L^\infty$. We refer the reader to \cite{S1} for more detailed information in this regard. Here
	we only cite the following result from \cite{S1}.
	\begin{lemma}\label{conl}
		Suppose that $\Omega$ is an open set in $\mathbb{R}^N$, and that $u\in L^1_{\mbox{loc}}	(\Omega)$ satisfies
		$\Delta u\in \mathcal{H}^1_{\mbox{loc}}(\Omega)$. Then $u$ is locally a continuous function on $\Omega$ when
		$N=2$.
	\end{lemma}
	
	\section{Main results} In this section, we first develop a couple of regularity results for \eqref{e1}-\eqref{e4} in two space dimensions.  Our main theorem is then established a consequence of these results. 
	
	The reason our first two propositions in this section are proved for the time-dependent case is
	%\eqref{e1}-\eqref{e4}
	not only for the purpose of generality but also due to a private communication to us by P. Markowich
	stating that numerical experiments for \eqref{e1}-\eqref{e4} in the generality considered here show no signs of singular behavior in solutions.  This suggests possible existence of a classical solution to the problem. Unfortunately,
	the method we have  developed here relies on Lemma \ref{conl}, and there is currently no suitable parabolic version of this lemma. As a result, our proof of the main theorem can not be extended to the time-dependent case.  Thus it remains to be seen that the numerical evidence mentioned earlier can be verified analytically.
	\begin{proposition}\label{pro1} Let (H1)-(H3) hold and $(m, p)$ be a weak solution to \eqref{e1}-\eqref{e4}.
		Assume that $N$=2. Then $p\in L^\infty(0,T; C_{\mbox{loc}}(\Omega))$.
	\end{proposition}
	\begin{proof}
		Fix a point $y$ in $\Omega$. Note that $m\in C([0,T]; \left(L^2(\Omega)\right)^N)$. For each $t\in (0, T]$ and each $r\in (0, R]$, where $R=\mbox{dist}(y,\partial\Omega)$, 
		%with $B_r(y)\subset\Omega$, 
		we consider the boundary value problem
		\begin{align}
		-\mbox{{div}}\left[(I+m(x,t)\otimes m(x,t))\nabla p_1\right]&=S(x)\ \ \ \mbox{in $B_r(y)$,}\label{do1}\\
		p_1&=0\ \ \ \mbox{on $\partial B_r(y)$.}\label{do2}
		\end{align}
		Even though the elliptic coefficients in \eqref{do1} may not be bounded above, we can easily infer from the proof of Lemma 2.3 in \cite{X4} that 
		this problem has a unique solution $p_1=p_1(x,t)$ in the following sense:
		\begin{enumerate}
			\item[(D4)]  $p_1\in W_0^{1,2}(B_r(y)),\ m\cdot\nabla p_1\in L^2(B_r(y))$;
			\item[(D5)] for each $\xi \in W_0^{1,2}(B_r(y))$ with $m\cdot\nabla \xi\in L^2(B_r(y))$ one has
			\begin{equation}
			\int_{B_r(y)}\left(\nabla p_1\nabla\xi+(m\cdot\nabla p_1)(m\cdot\nabla \xi)\right)dx=\int_{B_r(y)}S(x)\xi dx.
			\end{equation}
		\end{enumerate} Furthermore, we are in a position to assert from (H1) and Proposition 2.1 in \cite{LX} that there is a positive  number $c=c(N, q_0)$ such that
		\begin{equation}
		\sup_{B_r(y)}|p_1|\leq cr^{2-\frac{N}{q_0}}\left(\int_{B_r(y)}|S(x)|^{q_0}dx\right)^{\frac{1}{q_0}}.
		\end{equation}
		Here and in what follows $\sup$ (resp. $\inf$) means $\mbox{ess sup}$ (resp. $\mbox{ess inf}$). Consequently, 
		$p_0\equiv p-p_1$ 
		is the unique solution of 
		the boundary value problem
		\begin{align}
		-\mbox{{div}}\left[(I+m\otimes m)\nabla p_0\right]&=0\ \ \ \mbox{in $B_r(y)$,}\label{dz1}\\
		p_0(x,t)&=p(x,t)\ \ \ \mbox{on $\partial B_r(y)$}\label{dz2}
		\end{align}
		in the sense of (D4)-(D5) with an obvious modification to the boundary condition. That is,  we can decompose $p(x,t)$ into the sum of $p_0(x,t)$ and $p_1(x,t)$ on $B_r(y)$, or equivalently,
		\begin{equation}
		p=p_0+p_1\ \ \ \mbox{on $B_r(y)$}.
		\end{equation} Set $k_l=\sup_{\partial B_r(y)}p$. Then $(p_0-k_l)^+\in W_0^{1,2}(B_r(y))$ with $m\cdot\nabla(p_0-k_l)^+ \in L^2(B_r(y))$. Thus we can use it as a test function in \eqref{dz1}, thereby obtaining
		$$p_0\leq k_l \ \ \mbox{a.e. on $B_r(y)$.}$$
		In fact, we can further conclude that
		%It is not difficult to derive from \eqref{dz1} that 
		$p_0$ is a weakly monotone function \cite{M}, i.e.,
		\begin{equation}
		\sup_{\partial \Omega^\prime}p_0=\sup_{ \Omega^\prime}p_0\ \ \ \mbox{and }\ \ \ \inf_{\partial \Omega^\prime}p_0=\inf_{ \Omega^\prime}p_0
		\end{equation}
		for each sub-domain $\Omega^\prime$ of $B_r(y)$. By Morrey's inequality on spheres as formulated by Gehring \cite{G1}, we obtain
		\begin{equation}
		\mbox{osc}_{ B_r(y)}p_0\equiv\sup_{ B_r(y)}p_0-\inf_{ B_r(y)}p_0=\sup_{ \partial B_r(y)}p_0-\inf_{\partial B_r(y)}p_0\leq c\left(r\int_{\partial B_r(y)}|\nabla p_0|^2ds\right)^{\frac{1}{2}}.
		\end{equation}
		%where $c$ is a positive number. 
		Of course, the above inequality can also be established via an elementary calculus argument. Keeping the preceding estimates in mind,  we compute
		\begin{equation}
		\begin{split}
		\omega_r(y)\equiv &\ \mbox{osc}_{ B_r(y)}p\\
		%=\sup_{ B_r(y)}p-\inf_{ B_r(y)}p\\
		=&\sup_{x_1, x_2\in B_r(y)}(p(x_1,t)-p(x_2,t))\\
		=&\sup_{x_1, x_2\in B_r(y)}(p_0(x_1,t)-p_0(x_2,t)+p_1(x_1,t)-p_1(x_2,t))\\
		\leq &\ \mbox{osc}_{ B_r(y)}p_0+\mbox{osc}_{ B_r(y)}p_1\\
		\leq &  c\left(r\int_{\partial B_r(y)}|\nabla p_0|^2ds\right)^{\frac{1}{2}}+cr^{2-\frac{2}{q_0}}.
		\end{split}
		\end{equation}
		Remember that $q_0>1$. 
		Square both sides of the above inequality, divide through the resulting inequality by $r$, and integrate to
		obtain
		\begin{equation}\label{}
		\begin{split}
		\int_{r}^{R}\frac{\omega^2_s(y)}{s}ds\leq & c\int_{r}^{R}\int_{\partial B_s(y)}|\nabla p_0|^2dsds+c\left(R^{4-\frac{4}{q_0}}-r^{4-\frac{4}{q_0}}\right)\\
		\leq & c\int_{ B_R(y)}|\nabla p_0|^2dx+cR^{4-\frac{4}{q_0}}.
		\end{split}
		\end{equation}
		%where $R=\mbox{dist}(y,\partial\Omega)$. 
		%The last step above is due to . 
		Note that $\omega_r(y)$ is a decreasing function of $r$. We deduce from
		the preceding inequality that
		\begin{equation}\label{osc1}
		\omega^2_r(y)\ln\frac{R}{r}\leq c\int_{ B_R(y)}|\nabla p_0|^2dx+cR^{4-\frac{4}{q_0}}.
		\end{equation}
		Obviously, $p_0-p$ is a legitimate test function for \eqref{dz1}. Upon using it, we derive
		\begin{equation}
		\int_{ B_r(y)}|\nabla p_0|^2dx +\int_{ B_r(y)}|m\cdot\nabla p_0|^2dx\leq \int_{ B_r(y)}|\nabla p|^2dx+\int_{ B_r(y)}|m\cdot\nabla p|^2dx.
		\end{equation}
		This combined with \eqref{osc1}
		% and the fact that 
		yields
		\begin{equation}
		\omega_r(y)\leq \frac{c}{\sqrt{\ln\frac{R}{r}}}\equiv \sigma(r)\rightarrow 0\ \ \mbox{as $r\rightarrow 0$}.
		\end{equation}
		This finishes the proof.
	\end{proof}
	%Obviously, we also have 
	%while $p_1(x,t)$ solves the problem
	The continuity of $p$ in the space variables enables us to derive a local version of Proposition 2.1 in \cite{LX}. As we shall see, the key difference is that here $\beta$ does not have to be small. In fact, we shall establish that $|m|^\beta\in L^2(0,T; W^{1,2}_{\mbox{loc}}(\Omega)),\ |m|^\beta|\nabla p|^2\in L^1(0,T;L^1_{\mbox{loc}}(\Omega)) $ for each $\beta>0$.
	
	\begin{proposition}\label{pro2} Let the assumptions of Proposition \ref{pro1} hold.
		Fix  $K>0$ and define
		\begin{equation}
		v=(|m|^2-K^2)^++K^2.\label{vd}
		\end{equation}
		Then for each $\beta>0$ and  each $y\in\Omega$ there hold
		%exists $r_0\in (0, \mbox{dist}(y, \partial\Omega))$ such that
		\begin{equation}
		v^\beta\in L^2(0,T; W^{1,2}(B_{\frac{1}{2}r}(y))),\ \ \  v^\beta|\nabla p|^2\in L^1(B_{\frac{1}{2}r}(y)\times(0,T))
		%\ \ \mbox{for $r\in (0,r_0)$},
		\end{equation}	
		for some $r\in (0, \mbox{dist}(y, \partial\Omega))$ determined by the given data.
		%where 
		%
	\end{proposition}
	\begin{proof}
		%The title of your first subsection in section 2
		Let $K>0, \beta>0 $ be given and $v$ be defined as in \eqref{vd}. For $L>K$, define
		\begin{equation}\label{cutoff}
		\theta_L(s)=\left\{\begin{array}{ll}
		L^2&\mbox{if $s\geq L^2$,}\\
		s&\mbox{if $K^2<s< L^2$,}\\
		K^2 &\mbox{if $s\leq K$.}
		\end{array}\right.
		\end{equation}
		Set $v_L=\theta_L(|m|^2)$.
		%Without loss of generality, we may assume that $m\in L^\infty(\Omega_T)$. (Otherwise, we use a suitable %cut-off function for $|m|^2$.)
		%As we shall see, our approximate solutions are bounded.
		Fix $y\in\Omega$, pick  $r\in (0, \mbox{dist}(y, \partial\Omega))$, and select a smooth cutoff function $\xi: \mathbb{R}^N\rightarrow \mathbb{R}$
		satisfying 
		%C^\infty$ function $\xi$ on $\mathbb{R}^{N+1}$ with $$
		%and
		\begin{eqnarray*}
			0&\leq&\xi\leq 1,\\
			\xi&=& 1\ \ \ \mbox{on $B_{\frac{1}{2}r}(y)$},\\
			\xi&=& 0\ \ \ \mbox{off $B_{r}(y)$},\\	
			%	|	\partial_t\xi|&\leq&\frac{c}{r^2},\\
			|	\nabla\xi|&\leq&\frac{c}{r}.
		\end{eqnarray*}
		Then the function $v_L^\beta m\xi^2$ is a legitimate test function for \eqref{e2}. Upon using it, we arrive at
		\begin{eqnarray}
		\lefteqn{\frac{1}{2}\frac{d}{dt}\int_{B_r(y)} \int_{0}^{|m|^2}[\theta_L(s)]^\beta ds\xi^2 \ dx+D^2\int_{B_r(y)} v_L^\beta|\nabla m|^2\xi^2\ dx}\nonumber\\
		&&+\frac{D^2\beta}{2}\int_{B_r(y)} v_L^{\beta-1}|\nabla v_L|^2\xi^2\ dx+\int_{B_r(y)}|m|^{2\gamma}v_L^\beta\xi^2\ dx\nonumber\\
		&=&-2D^2\int_{B_r(y)} v_L^\beta\nabla m m\nabla\xi\xi\ dx+E^2\int_{B_r(y)} v_L^\beta\mnp^2\xi^2\ dx\nonumber\\
		&\leq& \varepsilon D^2\int_{B_r(y)} v_L^\beta|\nabla m|^2\xi^2\ dx+c(\varepsilon)\int_{B_r(y)} v_L^\beta| m|^2|\nabla\xi|^2\ dx\nonumber\\
		&&+E^2\int_{B_r(y)} v_L^\beta\mnp^2\xi^2\ dx,\label{n1}
		\end{eqnarray}
		where $\varepsilon>0$.
		In the derivation of the third term above,  we have used the fact that
		\begin{equation}\label{rs11}
		\nabla v_L=0\ \ \mbox{on the set where $|m|^2>L^2$ or $|m|^2<K^2$}.
		\end{equation}
		Also observe that $\nabla m=\nabla\otimes m$, and we still have
		%\begin{equation}
		$\nabla\left(\frac{1}{2}|m|^2\right)=\nabla m m$.
		%\end{equation}
		Set
		%\begin{equation}
		$\prt=\avint_{B_r(y)} p(x,t)dx$.
		%=\frac{1}{|B_r(y)|}\ib p(x,t)dx.
		%\end{equation}
		Note that $m\otimes m\nabla p=\mnp m$. Keep this in mind and use $v_L^\beta (p-\prt)\xi^2$ as a test function in \eqref{e1} to deduce
		\begin{eqnarray}
		\lefteqn{\int_{B_r(y)} v_L^\beta|\nabla p|^2 \xi^2\ dx +\int_{B_r(y)} v_L^\beta\mnp^2 \xi^2\ dx }\nonumber\\
		&=&-\int_{B_r(y)} \nabla p (p-\prt)\beta v_L^{\beta-1}\nabla v_L  \xi^2\ dx -\int_{B_r(y)} \nabla p (p-\prt) v_L^{\beta}2\nabla \xi  \xi\ dx\nonumber \\
		&&-\int_{B_r(y)} \mnp m(p-\prt) v_L^{\beta}2\nabla \xi  \xi\ dx+\int_{B_r(y)} S(x)v_L^\beta (p-\prt)  \xi^2\ dx \nonumber\\
		&&-\int_{B_r(y)} \mnp m(p-\prt)\beta v_L^{\beta-1}\nabla v_L  \xi^2\ dx \\
		&\leq & \varepsilon \int_{B_r(y)} v_L^{\beta}|\nabla p|^2  \xi^2\ dx+ c(\varepsilon)\beta^2\int_{B_r(y)} v_L^{\beta-2}(p-\prt)^2|\nabla v_L|^2 \xi^2\ dx\nonumber\\
		&& + c(\varepsilon)\int_{B_r(y)} v_L^{\beta}(p-\prt)^2|\nabla \xi|^2\ dx+\varepsilon\int_{B_r(y)} v_L^\beta\mnp^2 \xi^2\ dx \nonumber\\
		&& +c(\varepsilon)\beta^2\int_{B_r(y)} v_L^{\beta-2}|m|^2(p-\prt)^2|\nabla v_L|^2 \xi^2\ dx \nonumber\\
		&&+c(\varepsilon)\int_{B_r(y)} v_L^{\beta}|m|^2(p-\prt)^2|\nabla \xi|^2\ dx+\int_{B_r(y)} S(x)v_L^\beta (p -\prt) \xi^2\ dx ,\nonumber
		\end{eqnarray}
		where $\varepsilon>0$ is given as before.
		By virtue of \eqref{rs11}, we have that $v_L^{\beta-2}|m|^2|\nabla v_L|^2=v_L^{\beta-1}|\nabla v_L|^2$.
		Also, $v_L\geq K^2$ and $\max_{B_r(y)}(p-\prt)^2\leq \sigma^2(r)$. %Remember that $\beta\in (0,1)$. This gives $ v_L^{\beta-1}p^2\leq \|p\|_\infty^2K^{2(\beta-1)}$. 
		Choose $\varepsilon$ suitably small, multiply through the above inequality by $2E^2$, add the resulting inequality to 
		\eqref{n1}, and thereby obtain
		\begin{eqnarray}
		\lefteqn{\frac{d}{dt}\int_{B_r(y)} \int_{0}^{|m|^2}[\theta_L(s)]^\beta ds \xi^2\ dx +\int_{B_r(y)} v_L^\beta|\nabla m|^2 \xi^2\ dx }\nonumber\\
		&&+\beta\int_{B_r(y)} v_L^{\beta-1}|\nabla v_L|^2 \xi^2\ dx +\int_{B_r(y)}|m|^{2\gamma}v_L^\beta  \xi^2\ dx \nonumber\\
		&&+\int_{B_r(y)} v_L^\beta|\nabla p|^2 \xi^2\ dx +\int_{B_r(y)} v_L^\beta\mnp^2 \xi^2\ dx \nonumber\\
		&\leq&c\beta^2\frac{\sigma^2(r)}{K^2} \int_{B_r(y)}v_L^{\beta-1}|\nabla v_L|^2 \xi^2\ dx +c\beta^2\sigma^2(r)\int_{B_r(y)} v_L^{\beta-1}|\nabla v_L|^2 \xi^2\ dx \nonumber\\
		&& +c\int_{B_r(y)} v_L^{\beta}(p-\prt)^2|\nabla \xi|^2\ dx+c\int_{B_r(y)} v_L^{\beta}|m|^2(p-\prt)^2|\nabla \xi|^2\ dx\nonumber\\
		&&+c\int_{B_r(y)} v_L^{\beta}|m|^2|\nabla \xi|^2\ dx+\int_{B_r(y)} S(x)v_L^\beta (p-\prt)  \xi^2\ dx \nonumber\\
		&\leq&
		%c\beta^2\frac{\sigma^2(r)}{K^2} \int_{B_r(y)}v_L^{\beta-1}|\nabla p|^2 \xi^2\ dx +
		c\beta^2\sigma^2(r)\int_{B_r(y)} v_L^{\beta-1}|\nabla v_L|^2 \xi^2\ dx +c\int_{B_r(y)} v^{\beta+1}|\nabla \xi|^2\ dx \nonumber\\
		&&
		+c\int_{B_r(y)} |S(x)|v^\beta \xi^2\ dx+c.
		\end{eqnarray}
		%	Integrate this inequality with respect to t, 
		%	|m|^2\leq v,
		Here we have used the fact that $p\in L^\infty(\ot)$. This is due to Proposition 2.1 in \cite{LX}. In view of Proposition \ref{pro1}, $\lim_{r\rightarrow 0}\sigma(r)=0$. We can choose $r$ sufficiently small so that the first term on the right-hand in the above inequality can be absorbed into the similar term on the left-hand side there.  Integrating the resulting inequality with respect to $t$ and then taking $L\rightarrow \infty$ yield the desired result. The proof is complete.
	\end{proof}
	
	The core of our development is the following proposition, whose proof is inspired by an argument in \cite{CLMSN},
	based upon important contributions due to M\"{uller} \cite{MU}. Also see Proposition 2.1 in \cite{E}, which has become known as the div-curl lemma.
	
	\begin{proposition}\label{pro3}Let the assumptions of Proposition \ref{pro1} hold and $(m, p)$ be a stationary weak solution to \eqref{e1}-\eqref{e3}. Then
		$|\nabla p|^2,\ \mnp^2\in \mathcal{H}^1_{\mbox{loc}}(\Omega)$.
	\end{proposition}
	%Before we begin to prove, 
	\begin{proof}
		%[Proof of Proposition \ref{pro3}]
		Note that this proposition would be a trivial consequence of Proposition 2.1 in \cite{E} if we had
		$\mnp m\in \left(L^2(\Omega)\right)^2$ due to \eqref{e1}. 
		Let $K_0\subset \Omega$ be compact. Pick $0<\varepsilon_0<\mbox{dist}(K_0, \partial\Omega)$. Fix $\phi\in\mathcal{T}$ with $\int_{\mathbb{R}^2}\phi dx>0$. For each $r\in (0, \varepsilon_0)$ and each $y\in K_0$ we use $\frac{1}{r^2}\phi^2(\frac{y-x}{r})(p-p_{y,r})$ as a test function in
		\eqref{e1} to obtain
		\begin{eqnarray}
		\lefteqn{\avint_{B_r(y)}|\nabla p|^2\phi^2(\frac{y-x}{r})dx+\avint_{B_r(y)}\mnp^2\phi^2(\frac{y-x}{r})dx}\nonumber\\
		&\leq&\frac{c}{r^2}\avint_{B_r(y)}|p-p_{y,r}|^2dx+\frac{c}{r^2}\avint_{B_r(y)}|m|^2|p-p_{y,r}|^2dx\nonumber\\
		&&+c\avint_{B_r(y)} |S(x)||p-p_{y,r}|dx\nonumber\\
		&\leq&\frac{c}{r^2}\avint_{B_r(y)}\left(1+|m|^2\right)|p-p_{y,r}|^2dx
		+c\avint_{B_r(y)} |S(x)||p-p_{y,r}|dx.\label{pe1}
		\end{eqnarray}
		%By virtue of Poincar\'{e}'s inequality (\cite{EG}, p. 141), we have
		%\begin{equation}
		%\frac{1}{r^2}\avint_{B_r(y)}|p-p_{y,r}|^2dx\leq c\left(\avint_{B_r(y)}|\nabla p|dx\right)^2\leq c\left[M(|\nabla %p|\chi_\Omega)(y)\right]^2,
		%\end{equation}
		Fix $1<s<q_0$, where $q_0$ is given as in (H1). We estimate
		\begin{eqnarray}
		\avint_{B_r(y)} |S(x)||p-p_{y,r}|dx&\leq& c\left(\avint_{B_r(y)} |S(x)|^sdx\right)^{\frac{1}{s}}\nonumber\\
		&\leq& c+\left(\avint_{B_r(y)} |S(x)|^sdx\right)^{\frac{q_0}{s}}\nonumber\\
		&\leq &c+\left[M(|S(x)|^s\chi_\Omega)(y)\right]^{\frac{q_0}{s}},
		\end{eqnarray}
		where $\chi_{\Omega}$ is the indicator function of $\Omega$ and $M(|S(x)|^s\chi_\Omega)(y)$ is the value of
		the maximal function $M(|S(x)|^s\chi_\Omega)$ at $y$.
		The fourth term in \eqref{pe1} is the most difficult one to handle. Fix $\sigma>1$. By Proposition \ref{pro2}, we have $|m|^{2\sigma} \in W^{1,2}_{\mbox{loc}}(\Omega)$. Pick a $C^\infty$ function $\xi$ on $\mathbb{R}^2$ satisfying 
		%It follows from the Sobolev extension theorem (\cite{EG}, p.135) that there is a function $u_m$ in $$W^{1,2}(\mathbb{R}^2)$ such that
		\begin{eqnarray*}
			\xi&=&1\ \ \mbox{on $\Omega_{\varepsilon_0}\equiv\{x\in\Omega:\mbox{dist}(x, \partial\Omega)\geq \varepsilon_0$}\},\\
			\xi&=& 0 \ \ \mbox{outside $\Omega_{\varepsilon_1}$ for some $\varepsilon_1\in (\varepsilon_0, \mbox{dist}(K_0, \partial\Omega))$}.
			%\|u_m\|_{W^{1,2}(\mathbb{R}^2)}&\leq& c\||m|^{2\sigma}\|_{W^{1,2}(\Omega_{\varepsilon_0})}.
		\end{eqnarray*} 
		Obviously, we have $|m|^{2\sigma}\xi\in W^{1,2}(\mathbb{R}^2)$. 
		%Since we are operating on $\Omega_{\varepsilon_0}$, we will understand $|m|^2\xi$ to be $|m|^2\xi$ in the subsequent calculations. 
		%Let $\alpha>0$ be given. 
		We can easily check that $|\nabla(|m|^2\xi+1)^\sigma|\in L^2(\mathbb{R}^2)$.
		%such that Extend $m$ to be zero outside $\Omega$ and the resulting function is still denoted by $m$. 
		%Obviously, we have $m\in  (W^{1,2}(\mathbb{R}^2))^2) $. 
		%Remember that the space dimension is $2$. We can derive from 
		An application of Poincar\'{e}'s inequality indicates that $(|m|^2\xi+1)^\sigma\in  VMO\subset  BMO$.  The reciprocal of the term $(|m|^2\xi+1)^\sigma$ is bounded and hence lies in $BMO$. By Lemma \ref{jnl}, we have $\left(|m|^2\xi+1\right)^\sigma\in A_{q}$ for each $q>1$. In particular, we take $q=\frac{2\sigma(s_1-1)}{s_1}$, where $s_1$ is a number
		from $\left(\frac{2\sigma}{2\sigma-1}, 2\right)$. Then there holds
		\begin{equation}\label{aqw}
		\sup_{B_r(y)}\left(\avint_{B_r(y)}(|m|^2\xi+1)^\sigma
		dx\left(\avint_{B_r(y)}(|m|^2\xi+1)^{-\frac{\sigma}{q-1}}
		dx\right)^{q-1}\right)\leq c.
		\end{equation}
		%can be estimated as follows:
		We estimate from Poincar\'{e}'s inequality that
		\begin{eqnarray}
		\lefteqn{\frac{1}{r^2}\avint_{B_r(y)}(|m|^2\xi+1)|p-p_{y,r}|^2dx}\nonumber\\
		&\leq&\frac{1}{r^2}\left(\avint_{B_r(y)}(|m|^2\xi+1)^\sigma
		dx\right)^{\frac{1}{\sigma}}\left(\avint_{B_r(y)}|p-p_{y,r}|^{\frac{2\sigma}{\sigma-1}}dx\right)^{\frac{\sigma-1}{\sigma}}\nonumber\\
		&\leq&c\left(\avint_{B_r(y)}(|m|^2\xi+1)^\sigma
		dx\right)^{\frac{1}{\sigma}}\left(\avint_{B_r(y)}|\nabla p|^{\frac{2\sigma}{2\sigma-1}}dx\right)^{\frac{2\sigma-1}{\sigma}}.\label{mt2}
		\end{eqnarray}
		Remember that $2>s_1>\frac{2\sigma}{2\sigma-1}$. Thus we have
		\begin{eqnarray}
		\lefteqn{\avint_{B_r(y)}|\nabla p|^{\frac{2\sigma}{2\sigma-1}}dx}\nonumber\\
		&=&\avint_{B_r(y)}|\nabla p|^{\frac{2\sigma}{2\sigma-1}}(|m|^2\xi+1)^{\frac{\sigma}{2\sigma-1}}(|m|^2\xi+1)^{-\frac{\sigma}{2\sigma-1}}dx\nonumber\\
		&\leq &\left(\avint_{B_r(y)}|\nabla p|^{s_1}(|m|^2\xi+1)^{\frac{s_1}{2}}\right)^{\frac{2\sigma}{s_1(2\sigma-1)}}\left(\avint_{B_r(y)}(|m|^2\xi+1)^{-\frac{\sigma s_1}{s_1(2\sigma-1)-2\sigma}}\right)^{1-\frac{2\sigma}{s_1(2\sigma-1)}}\nonumber\\
		&\leq &\left(M[|\nabla p|^{s_1}(|m|^2\xi+1)^{\frac{s_1}{2}}\chi_\Omega](y)\right)^{\frac{2\sigma}{s_1(2\sigma-1)}}\left(\avint_{B_r(y)}(|m|^2\xi+1)^{-\frac{\sigma }{q-1}}\right)^{\frac{q-1}{2\sigma-1}}.\label{mt3}
		\end{eqnarray}
		Use \eqref{mt3} in \eqref{mt2} and apply \eqref{aqw} to obtain
		\begin{eqnarray}
		\lefteqn{\frac{1}{r^2}\avint_{B_r(y)}(|m|^2\xi+1)|p-p_{y,r}|^2dx}\nonumber\\
		&\leq&\left(\avint_{B_r(y)}(|m|^2\xi+1)^\sigma
		dx\right)^{\frac{1}{\sigma}}\left(\avint_{B_r(y)}(|m|^2\xi+1)^{-\frac{\sigma }{q-1}}\right)^{\frac{q-1}{\sigma}}\nonumber\\
		&&\cdot\left(M[|\nabla p|^{s_1}(|m|^2\xi+1)^{\frac{s_1}{2}}\chi_\Omega](y)\right)^{\frac{2}{s_1}}\nonumber\\
		&\leq & c\left(M[|\nabla p|^{s_1}(|m|^2\xi+1)^{\frac{s_1}{2}}\chi_\Omega](y)\right)^{\frac{2}{s_1}}.
		\end{eqnarray}
		Collecting all the previous estimates in \eqref{pe1}, we arrive at
		\begin{eqnarray}
		\lefteqn{\sup_{0<r<\varepsilon_0}\left(\avint_{B_r(y)}|\nabla p|^2\phi^2(\frac{y-x}{r})dx+\avint_{B_r(y)}\mnp^2\phi^2(\frac{y-x}{r})dx\right)}\nonumber\\
		&\leq&c\left[M(|\nabla p|\chi_\Omega)(y)\right]^2+c\left(M[|\nabla p|^{s_1}(|m|^2\xi+1)^{\frac{s_1}{2}}\chi_\Omega](y)\right)^{\frac{2}{s_1}}\nonumber\\
		&&+\left[M(|S(x)|^s\chi_\Omega)(y)\right]]^{\frac{q_0}{s}}+c.
		\end{eqnarray}
		Integrate the above inequality over $K_0$ and keep in mind the inequality \eqref{md} and the fact that the exponents $\frac{2}{s_1},\ \frac{q_0}{s}$ on the right-hand side of the preceding inequality are both bigger than $1$ to derive
		\begin{eqnarray}
		\lefteqn{\int_{K_0}\sup_{0<r<\varepsilon_0}\left(\avint_{B_r(y)}|\nabla p|^2\phi^2(\frac{y-x}{r})dx+\avint_{B_r(y)}\mnp^2\phi^2(\frac{y-x}{r})dx\right)dy}\nonumber\\
		&\leq &c\int_{K_0}\left(\left[M(|\nabla p|\chi_\Omega)(y)\right]^2+\left(M[|\nabla p|^{s_1}(|m|^2\xi+1)^{\frac{s_1}{2}}\chi_\Omega](y)\right)^{\frac{2}{s_1}}\right)dy\nonumber\\
		&&+c\int_{K_0}\left[M(|S(x)|^s\chi_\Omega)(y)\right]^{\frac{q_0}{s}}dy+c\nonumber\\
		&\leq &c\int_{\mathbb{R}^2}\left(\left[M(|\nabla p|\chi_\Omega)(y)\right]^2+\left(M[|\nabla p|^{s_1}(|m|^2\xi+1)^{\frac{s_1}{2}}\chi_\Omega](y)\right)^{\frac{2}{s_1}}\right)dy\nonumber\\
		&&+c\int_{\mathbb{R}^2}\left[M(|S(x)|^s\chi_\Omega)(y)\right]^{\frac{q_0}{s}}dy+c\nonumber\\
		&\leq&c\int_{\Omega}|\nabla p|^2 dy+c\int_{\Omega}|\nabla p|^2(|m|^2\xi+1) dy+c\nonumber\\
		&\leq&c\int_{\Omega_{\varepsilon_1}}|\nabla p|^2|m|^2 dy+ c\leq c.\label{hse}
		\end{eqnarray}
		%Integrate the above inequality with respect to $t$ and note that the resulting integral on the right-hand side is bounded 
		The last step is due to Proposition \ref{pro2}. 
		%As observed in \cite{S1}, if we just use one $\phi$ with $\int \phi dx >0$ in the definition of the Hardy space, the resulting definition is equivalent to the original one. 
		This implies the desired result. The proof is complete.
	\end{proof}

	We are ready to prove the main theorem.
	\begin{proof}[Proof of the main theorem] If $m$ is time-independent, then \eqref{e2} becomes
		\begin{equation}\label{sta1}
		-D^2\Delta m=E^2\mnp\nabla p-|m|^{2(\gamma-1)}m \ \ \mbox{in $\Omega$.}
		\end{equation}
		We can easily deduce from Proposition \ref{pro3} that
		\begin{equation}
		\mnp\nabla p\in\mathcal{H}^1_{\mbox{loc}}(\Omega).
		\end{equation}	
		This together with Lemma \ref{conl} implies that $m$ is locally continuous. Subsequently,
		%Therefore, 
		%a suitable application of
		a result 
		%we can apply a local version of the theorem
		%a well-known result 
		in (\cite{R}, p. 82) becomes applicable, and upon using it, we arrive at
		%leads to 
		%\eqref{e1}, thereby obtaining
		\begin{equation}
		|\nabla p|\in L^q_{loc}(\Omega)\ \ \mbox{for each $q>1$.} 
		\end{equation}
		(More general results of this nature can be found in \cite{DI}.)
		It immediately follows from (H2) 
		%from whence follows 
		that the right-hand side of \eqref{sta1} 
		has the same integrability as $|\nabla p|$.
		%also lies in the preceding function space. 
		Consequently, we  can appeal to a local version of the Calderon-Zygmund inequality to derive
		\begin{equation*}
		m\in \left(W^{2,q}_{loc}(\Omega)\right)^N\ \ \mbox{for each $q>1$.} 
		\end{equation*}
		Thus $m\in \left(C^{1, \delta}_{\mbox{loc}}(\Omega)\right)^N$ for some $\delta\in (0,1)$. 
		%Remember that $p$ satisfies \eqref{e1}, and hence
		Now we are in a position to invoke the classical Schauder  estimates for \eqref{e1}. To be specific, we can derive from a local version of Theorem 6.13 in \cite{GT}
		%  , from whence follows 
		that $p\in C^{2,\delta}_{loc}(\Omega)$.
		% for some $\delta\in (0,1)$. 
		%Recall that , and thus
		It can easily be inferred from (H2)  that the last term in \eqref{sta1} is locally H\"{o}lder continuous. Use the  Schauder  estimates  for \eqref{sta1} to get $m\in \left(C^{2,\delta_0}_{loc}(\Omega)\right)^N$ for some $\delta_0\in (0,1)$.
		%This, in turn, says  $m$. 
		%The standard bootstrap argument
		%The desired result follows. 
		The proof is complete.
	\end{proof}

	\medskip
	% The data information below will be filled by AIMS editorial staff
	Received xxxx 20xx; revised xxxx 20xx.
	\medskip
	
\end{document}